\documentclass[11pt,a4paper,reqno]{amsart}
\pdfoutput=1
\usepackage[hidelinks]{hyperref}
\usepackage[english]{babel}
\usepackage[top=2.5cm,right=3cm,left=3cm,bottom=2.5cm,footskip=20pt]{geometry}
\usepackage{amssymb,tikz,url,boxedminipage,interval}
\usepackage{blkarray,wrapfig}
\usepackage{caption}
\allowdisplaybreaks[4]
\newtheorem{thm}{Theorem}
\newtheorem{lemma}[thm]{Lemma}
\newtheorem{prop}[thm]{Proposition}
\newtheorem{rem}[thm]{Remark}
\newtheorem{cor}[thm]{Corollary}
\newtheorem{df}[thm]{Definition}
\newtheorem{ex}[thm]{Example}

\newcommand{\R}{\mathbb{R}}
\newcommand{\C}{\mathbb{C}}

\newcommand{\Z}{\mathbb{Z}}
\newcommand{\tr}{\mathrm{Tr}}
\newcommand{\inner}[1]{\left<#1\right>}

\makeatletter
\@namedef{subjclassname@2020}{%
\textup{2020} Mathematics Subject Classification}
\makeatother

\title{Tolerance relations and quantization}

\date{May 2022}

\author[F.~D'Andrea, G.~Landi, F.~Lizzi]{Francesco D'Andrea, Giovanni Landi, Fedele Lizzi}

\address[F.~D'Andrea]
{\begin{tabular}[t]{l}
Universit\`a di Napoli Federico II
and I.N.F.N.~Sezione di Napoli \\
Complesso MSA, Via Cintia, 80126 Napoli, Italy
\end{tabular}}
\email{francesco.dandrea@unina.it}

\address[G.~Landi]
{\begin{tabular}[t]{l}
Universit\`a di Trieste, Via A. Valerio, 12/1, 34127 Trieste, Italy \\
Institute for Geometry and Physics (IGAP) Trieste, Italy \\
I.N.F.N.~Sezione di Trieste, Italy
\end{tabular}}
\email{landi@units.it}

\address[F.~Lizzi]
{\begin{tabular}[t]{l}
Universit\`a di Napoli Federico II and I.N.F.N.~Sezione di Napoli \\
Complesso MSA, Via Cintia, 80126 Napoli, Italy
\\[3pt]
Institut de Ci{\'e}ncies del Cosmos, Universitat de Barcelona
\end{tabular}
}
\email{fedele.lizzi@na.infn.it}



\begin{document}

\begin{abstract}
It is well known that ``bad'' quotient spaces (typically: non-Hausdorff) can be studied by associating to them the groupoid C*-algebra of an equivalence relation, that in the ``nice'' cases is Morita equivalent to the C*-algebra of continuous functions vanishing at infinity on the quotient space. It was recently proposed in 
\cite{CvS21} that a similar procedure for relations that are reflexive and symmetric but fail to be transitive (i.e.~\emph{tolerance relations}) leads to an operator system. In this paper we observe that such an operator system carries a natural product that, although in general non-associative, arises in a number of relevant examples. We relate this product to truncations of (C*-algebras of) topological spaces, in the spirit of \cite{DLM14}, discuss some geometric aspects and a connection with positive operator valued measures.
\end{abstract}

\maketitle

\vspace*{-5mm}

\tableofcontents

\vspace*{-5mm}

\section{Introduction}

\setlength{\parskip}{1ex}

A standard construction in Noncommutative Geometry is that of the C*-algebra associated to an equivalence relation $R$ on a locally compact Hausdorff space $X$: one can construct the groupoid $\mathcal{G}$ of the relation $R$ and, under some technical conditions --- for example if $\mathcal{G}$ is {\'e}tale ---, one can construct the convolution algebra of $\mathcal{G}$ and complete it to a C*-algebra (see e.g.~\cite{SSW20}).
A source of examples comes from group actions: if $G$ is a discrete group with a continuous action on $X$, the action groupoid (corresponding to the relation ``being on the same orbit'') is {\'e}tale; if $X$ is compact, the full groupoid C*-algebra is isomorphic to the crossed product $C(X)\rtimes G$; if, in addition, the action is free and proper --- which means that $X/G$ is a compact Hausdorff space, and even a smooth manifold if $X$ is a smooth manifold and the action is smooth --- then the above crossed product is strongly Morita equivalent to the commutative C*-algebra $C(X/G)$ (this is a special case of, e.g., ``situation 2'' in \cite{Rie82}).
Applications of this construction cover foliations, orbifolds, tilings of the plane, dynamical systems arising in number theory such as the Bost-Connes system, just to name a few \cite{Con94,CM08}.
Notice that associativity of the groupoid partial composition law, and then of the convolution product, is a consequence of transitivity of $R$.

It is natural to wonder how much of the above picture can be generalized to relations $R$ that are reflexive and symmetric but not transitive, i.e.~those that in the literature are called ``tolerance relations''. It was observed by A.~Connes and W.D.~van Suijlekom in \cite{CvS20} that from any tolerance relation one can get an \emph{operator system}, which they propose as a starting point for a generalization of one of the main ingredients of Noncommutative Geometry, namely \emph{spectral triples} \cite{Con94,Lan97,GVF01}. The connection between tolerance relations and operator systems was then developed in \cite{CvS21}.
It turns out that these operator systems possess, in fact, the additional structure of an algebra. However, in general, this algebra would not be associative.

The aim of this paper is to study the convolution algebra of a tolerance relation. We start, in \S\ref{sec:2}, with a short introduction to tolerance relations and present a few motivating examples. In \S\ref{sec:3} we discuss the construction of convolution algebras of tolerance relations: we study finite-dimensional examples, infinite-dimensional ones associated to {\'e}tale tolerance relations, and special actions of magmas on sets replacing, in this picture, the above-mentioned action groupoids. We will also relate this construction to truncations of (C*-algebras of) topological spaces, in the spirit of \cite{DLM14}. In \S\ref{sec:5} we focus on state spaces and discuss some geometric aspects. In particular, we will give an explicit description of the set of pure states of the finite-dimensional systems from \S\ref{sec:3}. Finally, \S\ref{sec:6} is devoted to a curious example of a positive operator valued measure (POVM) that comes from a natural tolerance relation on the circle.

\section{Tolerance relations}\label{sec:2}

It was Poincar{\'e} who first observed, in a discussion about mathematical and physical continuum,
that ``equality'' in the real world is not always transitive, due to potential measurement errors \cite[Chap.~2]{PoiBook}. A beautiful concrete example with Wolfram Mathematica was suggested in \cite{McC15}. The code is reported in Table \ref{tab:math}: here we define $b:=a+\varepsilon$ and $c:=b+\varepsilon$ with $\varepsilon$ sufficiently small. Due to the finite storage capacity for numbers, the computer considers $a=b$ and $b=c$, but recognizes that $a\neq c$. This is exactly the behaviour one would expect in a physical experiment when using an instrument with finite resolution.

\begin{table}[t]
\begin{boxedminipage}{0.9\textwidth}
\small
\begin{verbatim}
eps = 50*$MachineEpsilon;
a = 1;
b = a + eps;
c = b + eps;
{a == a, b == b, c == c, a == b, b == a, b == c, c == b, a == c, c == a}

(* Out:  {True, True, True, True, True, True, True, False, False}  *)
\end{verbatim}
\end{boxedminipage}
\medskip
\caption{Equality in Wolfram Mathematica.}\label{tab:math}
\vspace{-10pt}
\end{table}

A relation that is reflexive and symmetric, but not necessarily transitive, is called a \emph{tolerance relation}.
The name was coined in \cite{Zee62}, while a mathematical theory was later developed in \cite{Sos86}.

The one in Table \ref{tab:math} is an example of ``proximity'' relation: on a metric space $(X,d)$ we can fix an $\varepsilon>0$ and consider the relation $R\subset X\times X$ defined by
\begin{equation}\label{eq:proximity}
(x,y)\in R \iff d(x,y)<\varepsilon .
\end{equation}
Such a relation in general fails to be transitive, except in special cases: for example if $d$ is an ultrametric (which means $d(x,z)\leq\max\{d(x,y),d(y,z)\}\;\forall\;x,y,z\in X$), then \eqref{eq:proximity} is an equivalence relation.

Another important example is the tolerance relation associated to a cover. Let $X$ be a set and $\mathcal{U}$ a collection of subsets covering $X$. The relation $R\subset X\times X$ given by
\begin{equation}\label{eq:covering}
(x,y)\in R\iff  \left\{ \exists\;S\in\mathcal{U}: x,y\in S \right\}
\end{equation}
is clearly a tolerance relation. From a physical point of view, we can imagine that this models an experiment where two points contained in the same set of the cover are not distinguishable one from the other. For $X=\R^n$ with Euclidean distance, the example \eqref{eq:proximity} corresponds to the cover of $\R^n$ by open balls of diameter $\varepsilon$.

It is worth mentioning that the idea of using finite covers to approximate a topological space is not new:
it can be traced back to Sorkin and his work on posets \cite{Sor91}, and to \cite{BBELLST96,LL99} in the framework of Noncommutative Geometry.

A notable example of tolerance relation in physics is \emph{causality} in special relativity: the set of pairs $(x,y)$ of points in Minkowski space that are causally connected, that means $x-y$ is either timelike or lightlike, is reflexive and symmetric but not transitive. This is illustrated pictorially in Figure \ref{fig:causality}, where $(x,y),(y,z)\in R$ and $(x,z)\notin R$.

\begin{figure}[t]
\begin{center}
\begin{tikzpicture}[font=\small]

\begin{scope}[xshift=-1.5cm,scale=0.3,very thin]

\draw[fill=cyan!40] (0,-2) ellipse (1.3 and 0.3);
\fill[cyan!40] (-1.3,1.98) -- (0,0) -- (1.3,1.98) -- cycle;
\fill[cyan!40] (-1.3,-1.98) -- (0,0) -- (1.3,-1.98) -- cycle;
\draw (-1.3,1.98) -- (1.3,-1.98) (-1.3,-1.98) -- (1.3,1.98);
\draw[fill=cyan!10] (0,2) ellipse (1.3 and 0.3);
\filldraw (0,0) circle (0.1) node[left=1pt] {$x$};
\draw[dashed,black!70!cyan] (1.3,-2) arc(0:180:1.3 and 0.3);

\end{scope}

\begin{scope}[xshift=1.5cm,scale=0.3,very thin]

\draw[fill=cyan!40] (0,-2) ellipse (1.3 and 0.3);
\fill[cyan!40] (-1.3,1.98) -- (0,0) -- (1.3,1.98) -- cycle;
\fill[cyan!40] (-1.3,-1.98) -- (0,0) -- (1.3,-1.98) -- cycle;
\draw (-1.3,1.98) -- (1.3,-1.98) (-1.3,-1.98) -- (1.3,1.98);
\draw[fill=cyan!10] (0,2) ellipse (1.3 and 0.3);
\filldraw (0,0) circle (0.1) node[right=1pt] {$z$};
\draw[dashed,black!70!cyan] (1.3,-2) arc(0:180:1.3 and 0.3);

\end{scope}

\begin{scope}[yshift=-3cm]

\draw (50:3.5) -- (0,0) -- (130:3.5);
\draw[dashed] (50:4) -- (50:3.5) (130:3.5) -- (130:4);
\filldraw (0,0) circle (0.03) node[below=1pt] {$y$};

\end{scope}

\end{tikzpicture}
\end{center}
\vspace*{-15pt}
\caption{Causality is not transitive.}\label{fig:causality}
\end{figure}
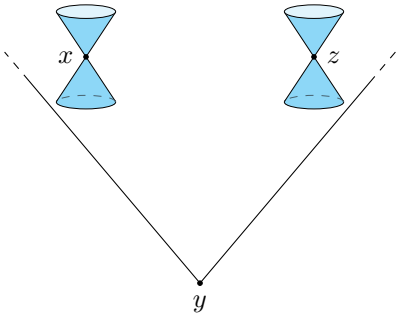

Finally, we mentioned that a source of examples of equivalence relations is given by group actions on sets. Given an action $\alpha:G\times X\to X$ of a group $G$ on a set $X$, the property
\begin{equation}\label{eq:lackproperty}
\alpha_g(\alpha_h(x))=\alpha_{gh}(x)\quad\forall\;g,h\in G,x\in X,
\end{equation}
ensures that the relation $R:=\{(x,\alpha_gx):g\in G,x\in X\}$ on $X$ is transitive. There are a number of examples, however, where the condition \eqref{eq:lackproperty} is not satisfied.

Suppose, for example, that we have a morphism
from a group $G$ to the group $$\mathrm{Out}(A):=\mathrm{Aut}(A)/\mathrm{Inn}(A)$$ of outer automorphisms of some unital C*-algebra $A$. In the Noncommutative Geometry approach to particle physics, 
when $A$ is almost commutative,
this corresponds to an action of $G$ as diffeomorphisms of the underlying manifold (cf.~\cite{CM08,vS15} and references therein, or the recent
physically-oriented review \cite{DKL19}). In such a situation, we can lift the above map to a map $\alpha:G\to\mathrm{Aut}(A)$ by picking up a representative for each class in $\mathrm{Out}(A)$. Such a map $\alpha$, however, will be only an action modulo inner automorphism.
If $A$ has only trivial inner automorphisms, given by the adjoint action of the unitary group $U(A)$, then
\begin{equation}\label{eq:outer}
\alpha_g\circ\alpha_h=\mathrm{Ad}_{f(g,h)}\circ\alpha_{gh}
\end{equation}
for all $g,h\in G$ and for a suitable function $f:G\times G\to U(A)$.
This reduces to \eqref{eq:lackproperty} if $f$ has image in the center of $A$, and was one of the motivating examples in \cite{BHM06} to develop a theory of non-associative crossed products.

Another class of examples where the condition \eqref{eq:lackproperty} fails is given by group \emph{quasi-actions}, that are a natural coarse generalization of isometric group actions:
here \eqref{eq:lackproperty} is replaced by the weaker condition
$$
d\big(\alpha_g\alpha_h(x),\alpha_{gh}(x)\big)\leq\varepsilon\quad\forall\;g,h\in G,x\in X,
$$
where $\varepsilon>0$ is a constant and $d$ is a distance on $X$ (see e.g.~\cite{MSW03}).

In \S\ref{sec:magmas} we will discuss a general setting where one can pass from an action to a tolerance relation.

\section{From relations to algebras}\label{sec:3}

In this section we mimic the construction of the convolution algebra of an {\'e}tale groupoid associated to an equivalence relation and explain how to pass from a tolerance relation to a complex $*$-algebra, possibly non-associative.
We will then study two special cases: tolerance relations on finite sets, and coming from suitable actions of magmas on sets.

Non-associative algebras naturally arise in the framework of cochain quantization, which consists in constructing a new product on
a module algebra $A$ over an Hopf algebra $H$ using an invertible element $F\in H\otimes H$. See e.g.~\cite{AM99} for a description of octonions as a cochain quantization of the group algebra of the group $(\Z/2\Z)^3$.
A similar construction was used in \cite{DF15} to describe the modules of ``sections of line bundles'' on the noncommutative torus as $U(1)$-spectral subspaces of a suitable non-associative formal deformation of the algebra of smooth functions on a Heisenberg manifold.
In deformation quantization, non-associative algebras arise if one tries to quantize a manifold $M$ equipped with a $2$-form $\omega$ which is non-degenerate but not closed, or more generally a manifold $M$ equipped with a skew-symmetric bi-derivation $\{\,,\,\}:C^\infty(M)\times C^\infty(M)\to C^\infty(M)$ which does not satisfy Jacobi identity; a similar phenomenon is encoutered in string theory in the presence of a B-field background \cite{Sza17}.

Let $X$ be a topological space and $R\subset X\times X$ a tolerance relation. We will say that $R$ is \emph{{\'e}tale} if there exists a topology on $R$ such that the projection $\mathrm{pr}_1:R\ni (x,y)\mapsto x\in X$ onto the first component is a local homeomorphism.\footnote{If $R$ is an equivalence relation, this simply means that the groupoid of the equivalence relation is {\'e}tale.}
We stress that the topology on $R$ is not necessarily the relative topology coming from the product topology on $X\times X$. In fact, it is usually finer \cite{GPS04}. See also Example \ref{ex:finer}.

\begin{lemma}
If $R$ is Hausdorff and {\'e}tale, the space $C_c(R)$ of compactly-supported complex-valued continuous functions on $R$ is a $*$-algebra with product and involution given by:
\begin{align}
(f\star g)(x,z) &:=\sum_{y\in X:(x,y),(y,z)\in R}f(x,y)g(y,z) \label{eq:prodgroupoid} \\
f^*(x,z) &:=\overline{f(z,x)} \label{eq:invgroupoid}
\end{align}
for all $(x,z)\in R$.
\end{lemma}

\begin{proof}
We wish to show that only finitely many elements in the sum \eqref{eq:prodgroupoid} are non-zero, so that the product is well-defined. We can rewrite the right hand side of \eqref{eq:prodgroupoid} as a sum of $f(\xi)g(\eta)$ over all 
$\xi,\eta\in R$ such that $\mathrm{pr}_1(\xi)=x$, $\mathrm{pr}_2(\eta)=z$ and $\mathrm{pr}_2(\xi)=\mathrm{pr}_1(\eta)$.
It is then enough to show that the set
\begin{equation}\label{eq:interS}
(\mathrm{pr}_1)^{-1}(x)\cap\mathrm{supp}(f)
\end{equation}
is finite. But local homeomorphisms have discrete fibers. Thus \eqref{eq:interS} is the intersection of the closed discrete subset
$(\mathrm{pr}_1)^{-1}(x)$ of $R$ and the closed compact subset $\mathrm{supp}(f)$ of $R$ (since $R$ is Hausdorff). Thus, \eqref{eq:interS} is a discrete compact space in the topology induced from $R$, hence finite.
\end{proof}

A neutral element for the product is given by the characteristic function of the diagonal $\Delta:=\{(x,x) :x\in X\}$, which is continuous if{}f $\Delta$ is clopen in $R$. However, since the diagonal map $X\to R$ is in general not continuous, even if $X$ is compact such a characteristic function may be not compactly supported (see Example \ref{ex:finer}).

\begin{ex}\label{ex:finer}
Let $X:=\interval{-1}{1}$ and $R$ be the equivalence relation generated by $x\sim -x$. Thus, $R$ is a cross. With 
the topology induced by $X\times X$, the projection onto the first axis is not a local homeomorphism (the cross is not locally Euclidean).
On the other hand we can write $R=(R\smallsetminus\{O\})\sqcup \{O\}$, put on $R\smallsetminus\{O\}$ and $\{O\}$ the subspace topology from $\R^2$, and on $R$ the disjoint union topology. With such a topology, $R$ is {\'e}tale.

Observe that the diagonal map $X\to R$, $x\mapsto (x,x)$, is not continuous, since $\{O\}$ is closed in $R$ but $\{0\}$ is not closed in $X$.

The diagonal $\Delta$ is clopen but not compact in $R$, since $\Delta\smallsetminus\{O\}$ is closed in $\Delta$ and non-compact. Hence the characteristic function of the interval is continuous but not compactly supported, and the convolution algebra is not unital (even if $X$ is compact).
\end{ex}

\begin{ex}\label{ex:finite-dim}
Let $X$ be a discrete space and $R$ any discrete tolerance relation on $X$. For $(i,j)\in R$,
let $E_{ij}$ be the function on $R$ that is $1$ at the site $(i,j)$ and zero everywhere else.
These functions form a basis of $C_c(R)$ (compactly means finitely supported in the present case). In such a basis, the product is explicitly given by:
\begin{equation}\label{eq:matrix}
E_{ij}\star E_{kl}=\bigg\{\!\begin{array}{ll}
\delta_{jk}E_{il} & \text{if }(i,l)\in R \\
0 & \text{otherwise}
\end{array}
\end{equation}
The algebra is unital if and only if $X$ is a finite set (i.e.~the diagonal $\Delta$ is compact), with unit given by
$$
1=\sum_{i\in X}E_{ii} \;.
$$
\end{ex}

When $R$ is an equivalence relation, it is well known that the convolution product is associative, and one can complete the convolution algebra using an injective bounded \mbox{$*$-representation}. However, if $(A,\star)$ is not associative, there is no injective homomorphism $(A,\star)\to \mathcal{B}(H)$ into bounded operators on a Hilbert space $H$. There is no good notion of ``representation'' that works for arbitrary non-associative algebras.

Observe that the definition of convolution product makes sense if we replace complex-valued functions by functions with values in any field $\mathbb{K}$. However, for the sake of simplicity, we will only study the case $\mathbb{K}=\C$.

\subsection{Finite-dimensional case}\label{sec:finitedim}
It is useful to have a pictorial representation of tolerance relations. Observe that $R\subset X\times X$ is a tolerance relation if and only if $\Gamma:=(X,R\smallsetminus \Delta)$ is an undirected graph with no loops. Of course, from any undirected graph with no loops we get a tolerance relation by adding a loop to each vertex.
Thus for example the graph:
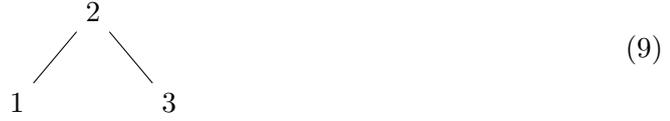
\begin{equation}\label{eq:exgraph}
\begin{tikzpicture}[baseline=(current bounding box.center)]

\node (a) at (0,1.2) {$2$};
\node (b) at (-1,0) {$1$};
\node (c) at (1,0) {$3$};

\path[-] (a) edge (b) (a) edge (c);

\end{tikzpicture}
\end{equation}
represents the relation on $X=\{1,2,3\}$ with $1\sim 2$ and $2\sim 3$, but $1\not\sim 3$ (where, as costumary, we write $i\sim j$ to mean that $(i,j)\in R$).

A tolerance relation is transitive if, whenever there is an edge between $x$ and $y$ and between $y$ and $z$, there is also an edge between $x$ and $z$. One immediately realizes by induction that:

\begin{rem}\label{rem:Rcon}
Let $X$ be a finite set. Then, a tolerance relation $R$ on $X$ is an equivalence relation if and only if each connected component of its graph is a complete graph.
\end{rem}

In the rest of this section, we will assume that $X:=\{1,\ldots,n\}$ is finite and that the topology on both $X$ and $R$ is the discrete one. The convolution product is then given by \eqref{eq:matrix} in the basis of delta functions. We will denote by $A(R)$ the algebra $C_c(R)$ with product \eqref{eq:matrix}, and call it the \emph{tolerance algebra} associated to $R$.

Observe that, if $(i,j)$ and $(k,l)$ are edges in different connected components of the graph of $R$, the product of the corresponding basis elements of $A(R)$ is zero. The algebra $A(R)$ is then a direct sum of algebras associated to connected components of the graph of $R$.
If $R$ is an equivalence relation with connected graph, which means $R=X\times X$ by Remark \ref{rem:Rcon},
if we identify $E_{ij}$ with the $n\times n$ matrix with $1$ in position $(i,j)$ and zero everywhere else,
we see that \eqref{eq:matrix} is just the usual matrix product, the involution \eqref{eq:invgroupoid} is the usual Hermitian conjugation, and $A(R)$ is $*$-isomorphic to $M_n(\C)$. In general,
if $R$ is an equivalence relation, $A(R)$  is isomorphic to a direct sum of matrix algebras (as it is well-known).

\begin{ex}\label{ex:one}
Let us denote by $\mathfrak{A}_3$ the tolerance algebra of the relation \eqref{eq:exgraph}. We think of its elements as matrices
$a=(a_{ij})\in M_3(\C)$ with $a_{13}=a_{31}=0$, and the product of two such elements is the matrix product composed with the natural projection $M_3(\C)\to \mathfrak{A}_3$ (which kills the $(1,3)$ and $(3,1)$ matrix elements).

The algebra $\mathfrak{A}_3$ is not power associative (which in particular means that it is neither commutative nor associative).  Indeed, if
$$
a:=\begin{bmatrix}
0 & 1 & 0 \\ 1 & 0 & 1 \\ 0 & 1 & 0
\end{bmatrix} ,
$$
then
$$
(a\star a)\star a=\begin{bmatrix}
0 & 1 & 0 \\ 2 & 0 & 2 \\ 0 & 1 & 0
\end{bmatrix}\neq
a\star (a\star a)=\begin{bmatrix}
0 & 2 & 0 \\ 1 & 0 & 1 \\ 0 & 2 & 0
\end{bmatrix} .
$$
\end{ex}

\begin{lemma}\label{lemma:subalgebra}
Let $R$ be a tolerance relation on $X=\{1,\ldots,n\}$.
Then, $R$ is an equivalence relation if and only if $A(R)$ has no subalgebra isomorphic to $\mathfrak{A}_3$.
\end{lemma}

\begin{proof}
If $R$ is an equivalence relation, then $A(R)$ cannot have a subalgebra isomorphic to $\mathfrak{A}_3$, since the former algebra is associative and the latter is not.
If $R$ is \emph{not} an equivalence relation, then there exists three distinct elements $x,y,z\in X$ such that $x\sim y$ and $y\sim z$ but $x\not\sim z$; the vector subspace of $A(R)$ spanned by the elements
$E_{xx},E_{yy},E_{zz},E_{xy},E_{yz},E_{yx},E_{zy}$ is a subalgebra isomorphic to $\mathfrak{A}_3$, with isomophism given by (with an abuse of notations): $x\to 1$, $y\to 2$, $z\to 3$.
\end{proof}

\begin{prop}\label{prop:Rpowerass}
Let $R$ be a tolerance relation on $\{1,\ldots,n\}$ and $A(R)$ the tolerance algebra of $R$. The following properties are equivalent:
\begin{itemize}
\item[(i)] $A(R)$ is associative;
\item[(ii)] $A(R)$ is power associative;
\item[(iii)] $R$ is an equivalence relation.
\end{itemize}
\end{prop}

\begin{proof}
If $R$ is not an equivalence relation, then $A(R)$ contains a subalgebra isomorphic to $\mathfrak{A}_3$ (Lemma \ref{lemma:subalgebra}). Since $\mathfrak{A}_3$ is not power associative, $A(R)$ cannot be power associative. Thus 
(ii) $\Rightarrow$ (iii). The implications (iii) $\Rightarrow$ (i) $\Rightarrow$ (ii) are obvious.
\end{proof}

It follows from the previous proposition that an algebra that is power associative but not associative cannot be obtained from the above construction.
Thus, for example, algebras obtained from $\R$ by a iterated Cayley--Dickson construction with dimension $\geq 8$ (octonion, sedenions, etc.) are not tolerance algebras. The same is true for split octionions.
Division algebras (such as quaternions) also cannot be obtained from the above construction: since $E_{ii}\star E_{jj}=0$ for all $i,j\in X$ with $i\neq j$, the algebra $A(R)$ is a division algebra if and only if $X$ is a singleton, i.e.~$A(R)=\C$.

\subsection{Magmas acting on sets}\label{sec:magmas}

Let $G$ be a set with a binary operation $\ast:G\times G\to G$ (the ``multiplication''), a unary operation $(-)^{-1}:G\to G$ (the ``right inversion''), and an element $1\in G$ (the ``right inverse''), such that:
\begin{equation}\label{eq:magma}
(g\ast h)\ast h^{-1}=g\ast 1=g\qquad\forall\;g,h\in G.
\end{equation}
We call $G$ a magma with a right inverse and unit.

\begin{ex}\label{ex:div}
Consider the set $\R\smallsetminus\{0\}$ with operation $x\ast y:=x/y$ and with standard inversion and unit. Observe that
$(x\ast y)\ast y^{-1}=xy^{-1}y=x$ even if $y\ast y^{-1}$ is not $1$.
\end{ex}

A right \emph{action} of $G$ on a set $X$ is a map $X\times G\to X$, $(x,g)\mapsto x\triangleleft g$, such that:
\begin{equation}\label{eq:onlyassume}
(x\triangleleft g)\triangleleft g^{-1}=x\triangleleft 1=x\qquad\forall\;g\in G,x\in X.
\end{equation}
Given such an action, the image $R(G,X)$ of the canonical map:
\begin{equation}\label{eq:Galois}
X\times G\to X\times X , \qquad (x,g)\mapsto (x,x\triangleleft g) ,
\end{equation}
is a tolerance relation. Indeed, $(x,x)=(x,x\triangleleft 1)\in R(G,X)$ for all $x\in X$, and if $(x,y)=(x,x\triangleleft g)\in R(G,X)$, then $(y,x)=(y,y\triangleleft g^{-1})\in R(G,X)$ as well.

The action $\triangleleft$ is called \emph{free} if \eqref{eq:Galois} is injective, \emph{transitive} if \eqref{eq:Galois} is surjective.

\begin{rem}\label{rem:magma}
The action of $G$ on itself by right multiplication is free if and only if $G$ is left-cancellative, i.e.~for all $x,g,h\in G$:
$$
x\ast g=x\ast h \quad\Longrightarrow\quad g=h .
$$
If $g^{-1}\ast (g\ast h)=h$ for all $g,h\in G$, then the action is both free and transitive. In Example \ref{ex:div}, the right action of $G$ on itself is both free and transitive.
\end{rem}

\begin{ex}
If $G$ is a Moufang loop \cite{Mou35}, the inverse properties of Moufang loops read: $g^{-1}\ast (g\ast h)=h=(h\ast g)\ast g^{-1}$ for all $g,h\in G$. They guarantee that $G$ satisfies \eqref{eq:magma} and that the right action of $G$ on itself is both free and transitive.
\end{ex}

If $G$ and $X$ are topological space and $G$ acts freely on $X$, the bijection $X\times G\to R(G,X)$
given by \eqref{eq:Galois} defines a topology on $R$ that we will call \emph{standard} (and in general is not the subspace topology from $X\times X$).
The projection $X\times G\to X$ onto the first component is a local homeomorphism if and only if $G$ is discrete.\footnote{Observe that no continuity assumption on the operations of $G$ is needed. When $G$ is discrete, they will automatically be continuous but the action on $X$ might still not be continuous.}
Thus:

\begin{rem}
Let $G$ and $X$ be topological spaces with $G$ acting freely on $X$. The tolerance relation $R(G,X)$ is {\'e}tale if and only if $G$ is discrete.
\end{rem}

A celebrated example where $G$ is a group and its action is a group action (in the standard sense) is the following.

\begin{ex}
Let $G:=\Z$ with action on $X:=\mathbb{S}^1:=\R/\Z$ by translations of a multiple of a fixed $\theta\in\R\smallsetminus\mathbb Q$. The associated relation $R(G,X)$ is then given by $x\sim y\iff x-y\in \theta\,\Z$. Since $\theta$ is irrational, the action is free and $R(G,X)$ with the topology of $\mathbb{S}^1\times\Z$ is {\'e}tale.
Let us identify $R$ with $\mathbb{S}^1\times\Z$. Every $f\in C_c(R)$ is a finite sum
\begin{equation}\label{eq:finitesum}
f=\sum\nolimits_jf_jV^j
\end{equation}
where $V^j$ is the delta function $V^j(k)=\delta_{j,k}$ and $f_j\in C(\mathbb{S}^1)$. The convolution product \eqref{eq:prodgroupoid} of two elements $fV^j$ and $gV^k$ is
\begin{equation}\label{eq:thetaprod}
fV^j\star gV^k(x,n)=f(x)g(x+j\theta)\delta_{j+k,n} \;.
\end{equation}
In particular $V^j\star V^k=V^{j+k}$, thus justifying the notations.

Let $U\in C(\mathbb{S}^1)$ be the function $U(x)=e^{2\pi\mathrm{i} x}$. One easily checks that
$$
V\star U=e^{2\pi\mathrm{i}\theta}U\star V.
$$
We recognize a dense $*$-subalgebra of the C*-algebra of the noncommutative torus \cite{Rie81}.
\end{ex}

One can construct many examples where $G$ is not a group. If $X\simeq S^7$ is the set of unit octonions (hence a Moufang loop) and $G\subset X$ a discrete subset containing $1$ and closed under inversion, then right multiplication of $X$ by $G$ defines a free action. In particular, $G$ can be a sub-loop, such as the loop of integer vectors in $S^7\subset\R^8$, or the set of basis vectors and their additive inverses, which are both Moufang loops of finite order, so that the associated tolerance relation is {\'e}tale.

In general, if $K$ is any group with a free right action $\triangleleft$ on a set $X$ and $G\subset K$ is a subset containing $1$ and closed under inversion, then the restriction of $\triangleleft $ to $G$ is a free action in our more general sense, and the associated relation may be non-trivial (not transitive) even if the multiplication in $K$ is associative. For example, if $K=S_3$ is the group of permutations of $X=\{1,2,3\}$ and $G\subset S_3$ is the subset containing the identity and the transpositions $1\leftrightarrow 2$ and $2\leftrightarrow 3$, then the associated relation is the one with graph \eqref{eq:exgraph}.

\subsection{Truncations}\label{sec:4}

It is a common attitude in Noncommutative Geometry to regard the Dirac operator of a spectral triple $(A,H,D)$ as a kind of inverse of the line element of Riemannian geometry:
$$
D\sim ds^{-1} .
$$
The precise meaning of this statement is explained for example in \cite{Con94}. Given a spectral triple, for example the one canonically associated to a Riemannian spin manifold $M$, one can imagine to implement a cut-off on large energies (small distances) by using the spectral projection $P_\Lambda:=\chi_{[-\Lambda,\Lambda]}(D)$, with $\Lambda>0$.
One can then replace the original spectral triple by the truncated one $(P_\Lambda AP_\Lambda,P_\Lambda H,P_\Lambda DP_\Lambda)$ and think of it as some kind of ``coarse graining'' of the manifold we started from. If $D$ has compact resolvent, e.g.~if the spectral triple is unital, the truncations will be finite-dimensional. In the example of a compact Riemannian spin manifold $M$, one obtains matrix geometries ``approximating'' $M$ (see, e.g., \cite{CvS20,DLM14}).

The issue of convergence of matrix geometries in the case of coadjoint orbits was studied in several seminal papers by Rieffel, see e.g.~the most recent one \cite{Rie21} and references therein. The metric convergence of truncations in the $\Lambda\to\infty$ limit were recently studied in \cite{vS21}: a sufficient condition, there, for convergence is the existence of a ``$C^1$-approximate order isomorphism'', see \cite[Def.~2]{vS21}. 
In the case of a compact Riemannian spin manifold $M$, when $A=C^\infty(M)$ and $D$ is the Dirac operator of the spin structure acting on the Hilbert space $H$ of $L^2$-spinors, it was shown in \cite{GS21} that one can choose a suitable set of ``localized'' states on each $P_\Lambda AP_\Lambda$ forming a sequence of metric spaces (equipped with the distance induced by the truncated Dirac operator) that converge to $M$ in the Gromov-Hausdorff sense.
A numerical study through computer simulations is in \cite{GS20}. The question of convergence of the full state space was solved in \cite{vS21} in the case of the circle, and is still open for a general (compact) Riemannian spin manifold.

A spectral projection as above defines a linear map $A\to P_\Lambda AP_\Lambda$ that allows one to put a non-associative product on its image, and is a special case of the following natural construction.

Let $B$ be a C*-algebra and $T:B\to B$ a linear map (our ``cut-off'' operator).
On $A:=\mathrm{Im}(T)$ we can define a product $\star$ by
\begin{equation}\label{eq:truncatedprod}
a\star b:=T(ab) \;\forall\;a,b\in A,
\end{equation}
where the one on the right hand side is the product of $B$, and get a (possibly non associative) algebra $(A,\star)$.
If $T(a^*)=T(a)^*$ for all $a\in A$, then $(A,\star)$ is a $*$-algebra.
If $B$ is commutative, then $(A,\star)$ is commutative (and then power associative).
We will call $(A,\star)$ a ``truncation'' of $B$.

Let us observe that for a truncation as above one has two notions of positivity: we could call an element $a\in A$ \emph{positive} if there exists a finite set of elements $b_1,\ldots,b_k\in A$ such that $a=\sum_{i=1}^kb_i\star b_i^*$, or we could call it positive if it is a positive element of the C*-algebra $B$. As shown in the next examples, these two notions do not always coincide.
If $T:B\to B$ is a \emph{positive} map, then $b\star b^*=T(bb^*)$ is positive in $B$ for all $b\in A$ and the former notion of positivity implies the latter.

A characteristic property of a truncation map defined by a ``sandwich'' with a spectral projection --- our motivating example --- is that it is \emph{idempotent}. That is:
\begin{equation}\label{eq:idempotent}
T(T(b))=T(b) \qquad\forall\;b\in B.
\end{equation}
From property \eqref{eq:idempotent} it follows that, in such a case, $B=A\oplus\ker T$.

Theorem II.6.10.11 of \cite{Bla06} provides a sufficient condition for \eqref{eq:truncatedprod} to be associative: if $T:B\to B$ is idempotent and a completely positive\footnote{Here completely positive means that the map \mbox{$\mathrm{id}\otimes T:M_k(\C)\otimes B\to M_k(\C)\otimes B$} is positive for all $k\geq 1$.} contraction, then $\star$ is associative and makes $A$ into a C*-algebra with its involution and with the norm inherited from $B$ (even if it is not, in general, a C*-subalgebra of $B$).

\begin{ex}\label{ex:14}
Let $B$ be a C*-algebra and $P\in B$ a projection. Then, the map $T:B\to B$ defined by $T(b):=PbP$ is a completely positive idempotent contraction. Observe, however, that in such a case associativity of \eqref{eq:truncatedprod} can be proved by a direct computation. Indeed, for every $a_1=T(b_1)$ and $a_2=T(b_2)$:
$$
a_1\star a_2=P(Pb_1P)(Pb_2P)P=(Pb_1P)(Pb_2P)=a_1a_2 .
$$
Thus, $\star$ is the restriction of the product of $B$
to the C*-subalgebra $A:=\mathrm{Im}(T)$. (See also Example 3.3 of \cite{DLM14}.)
\end{ex}

Example \ref{ex:14} is a special case of the following notion.
Let $A$ be a C*-subalgebra of a C*-algebra $B$. A linear map $T:B\to B$ with $\mathrm{Im}(T)=A$
is called a \emph{conditional expectation} from $B$ to $A$ if one of the two equivalent conditions is satisfied:
\begin{itemize}
\item[(i)] $T$ is idempotent with norm $1$;
\item[(ii)] $T$ is positive, idempotent and an $A$-bimodule map.
\end{itemize}
Both (i) and (ii) imply that $T$ is a completely positive contraction (see \cite[\S II.6.10]{Bla06}).
However, the setting is less general than that of \cite[Theorem II.6.10.11]{Bla06} since here we assume that $A$ is a C*-subalgebra of $B$, while the above mentioned theorem guarantees that $\star$ is associative, but it is not necessarily the restriction to $A$ of the product in $B$.
 
In the next examples, the map $T$ fails to be either positive or idempotent.

\begin{ex}\label{ex:previous}
Let $n\geq 2$ be fixed, $X:=\mathbb{S}^1$, let $u$ be the unitary generator of $B:=C(X)$, given by $u(t):=e^{\mathrm{i}t}$, and define:
\begin{equation}\label{eq:fourierpartial}
T(f):=\sum_{k=-n+1}^{n-1}u^k\widehat{f}(k)
=\frac{1}{2\pi}\sum_{k=-n+1}^{n-1}u^k\int_{0}^{2\pi}u^{-k}(t)f(t)dt .
\end{equation}
Observe that, in this example, the product \eqref{eq:truncatedprod} is power associative but not associative:
\mbox{$(u^*\star u^{n-1})\star u=u^{n-1}$} while
\mbox{$u^*\star (u^{n-1}\star u)=0$.} Moreover, in spite of power associativity, in general $f^{\star k}\neq T(f^k)$.
For $f:=u^{-n+1}+u^{n-1}\in A$, for example, one has
$$
T(f^3)=3f \neq f^{\star 3}= 2f \;.
$$
The map \eqref{eq:fourierpartial} satisfies \eqref{eq:idempotent}, but not positivity.
Take for example $n=2$ and $f=u+u^*+1$. Then $T(f^2)=3+2(u+u^*)$,
i.e.~$T(f^2)(t)=3+4\cos t$,  proving that the function $T(f^2)$ is not everywhere greater or equal than zero.
\end{ex}

Since the Fouries series of a continuous function on $\mathbb S^1$ may be divergent, a smart idea is to replace
the partial sum of the Fourier series \eqref{eq:fourierpartial} by the corresponding $n$-th Ces{\`a}ro sum. As shown in the next example, we loose the property \eqref{eq:idempotent} but gain positivity.

\begin{ex}\label{ex:opsystwo}
Let $n\geq 2$ be fixed, $X:=\mathbb{S}^1$, let $u$ be the unitary generator of $B:=C(X)$, given by $u(t):=e^{\mathrm{i}t}$, and define:
\begin{equation}\label{eq:Cesaro}
T(f):=\frac{1}{n}\sum_{j=1}^n\sum_{k=-j+1}^{j-1}u^k\widehat{f}(k) .
\end{equation}
From
$$
T(u^k)=\begin{cases}
\frac{n-|k|}{n}u^k & \text{if }|k|<n \\
0 & \text{if }|k|\geq n
\end{cases}
$$
one sees that the maps in \eqref{eq:fourierpartial} and \eqref{eq:Cesaro} have the same range, and that the latter is not idempotent. From the well-known formula in terms of Fej\'er kernel:
$$
T(f)(t)=\frac{1}{2\pi n}\int_{\mathbb S^1}f(t-\tau) \frac{\sin^2(n\tau/2)}{\sin^2(\tau/2)} d\tau ,
$$
it is evident that if $f$ is non-negative, then $T(f)$ is a non-negative function as well.

The product defined by the map \eqref{eq:Cesaro} is non-associative, as one can check, and
in general $f^{\star k}\neq T(f^k)$. For example if $f:=u^{-n+1}+u^{n-1}$, then
$T(f^3)=\frac{3}{n}f \neq f^{\star 3}= \frac{2}{n}f$.
\end{ex}

An interesting although simple observation is that every finite-dimensional tolerance algebra is a truncation of a matrix algebra.

\begin{ex}\label{ex:tol12}
Let $R$ be a tolerance relation on $X=\{1,\ldots,n\}$ and $T:M_n(\C)\to M_n(\C)$ be the map:
\begin{equation}\label{eq:trunmapfin}
T(b)=\sum_{(i,j)\in R}E_{ii}bE_{jj} \qquad\forall\;b\in B.
\end{equation}
Then, the product \eqref{eq:truncatedprod} coincides with the convolution product \eqref{eq:matrix} and $(A,\star)=A(R)$ is the tolerance algebra of \S\ref{sec:finitedim}.

Observe that such a $T$ is idempotent but, in general, not positive. Let us start with $n=3$ and $R$ the relation \eqref{eq:exgraph}. If:
$$
a:=\begin{bmatrix}
0 & 1 & 0 \\ 0 & 1 & 0 \\ 0 & 1 & 0
\end{bmatrix}, \text{ then \ }
a\star a^*=T(aa^*)=\begin{bmatrix}
1 & 1 & 0 \\ 1 & 1 & 1 \\ 0 & 1 & 1
\end{bmatrix}
$$
has determinant $-1$, hence $T(aa^*)$ is not a positive semidefinite matrix and $T:M_3(\C)\to M_3(\C)$ is not a positive map.

For a general tolerance relation $R$, arguing as in the proof of Lemma \ref{lemma:subalgebra}, one can show that $T$ is positive if and only if $R$ is an equivalence relation.

The vector space underlying the algebra $A(R)$ is called in \cite{KW12} the \emph{quantum relation} on $M_n(\C)$ associated to the relation $R$. See Sect.~1 and in particular Proposition 1.4 of \cite{KW12}.
\end{ex}

\section{State spaces}\label{sec:5}
\subsection{Operator systems}
Let $H$ be a complex Hilbert space. A concrete unital \emph{operator system} is a complex vector subspace $A\subset\mathcal{B}(H)$ containing the identity operator $1$ of $\mathcal{B}(H)$ and the adjoint of all its elements: $a^*\in A$ for all $a\in A$ (we will only consider \emph{unital} operator systems). This is very close to the notion of concrete \emph{order unit space}, which is a real vector subspace of the set of selfadjoint operators on $H$ containing the identity.
We refer to \cite{CvS20} for a recollection of properties of operator systems and their use in Noncommutative Geometry, and to \cite{Rie04} for order unit spaces and their use in the theory of compact quantum metric spaces.

If $P_n$ is the spectral projection of the Dirac operator $D$ of $\mathbb S^1$ on eigenspaces with eigenvalue $1\leq\lambda\leq n$, the truncation $C(\mathbb S^1)^{(n)}:=P_nC(\mathbb S^1)P_n$ is the operator system of $n\times n$ Toeplitz matrices ($n\times n$ matrices that are constant on descending diagonals) studied in \cite{CvS20,vS21}. A \emph{dual} operator system $C(\mathbb S^1)_{(n)}$ is given by the truncation of $C(\mathbb S^1)$ in Examples \ref{ex:previous}-\ref{ex:opsystwo} (see \cite{CvS20} for the exact meaning of operator system duality, and \cite{Far21} for the proof that in the example above there is a unital complete order isomorphism of operator systems).

The complex vector space underlying the algebra $A(R)$ in \S\ref{sec:finitedim} is obviously a finite-dimensional operator system. Note, however, that not every finite-dimensional operator system is of this form. The vector space $A(R)$ always contains all diagonal matrices, so for example, for $n\geq 2$, the operator system $C(\mathbb S^1)^{(n)}$ does not come from a tolerance relation.

If $A$ is an operator system, there is a notion of \emph{state} very much as for C*-algebras, as a positive bounded functional $\varphi:A\to\C$ satisfying $\varphi(1)=1$. The set of all states, denoted by $\mathcal{S}(A)$, is a convex subset of the Banach dual of $A$, compact in the weak* topology; its extremal points are called \emph{pure states}.
The set of all pure states of $A$ is denoted by $\mathcal{P}(A)$.

Even if we will not deal with convergence problems here, it is worth mentioning 
the result in \cite{vS21} that the state spaces of both operator systems $C(\mathbb S^1)^{(n)}$ and $C(\mathbb S^1)_{(n)}$, with the obvious metric induced by the Dirac operator of $\mathbb S^1$, converge to $\mathcal{S}(C(\mathbb S^1))$ in the Gromov-Hausdorff distance for $n\to\infty$.
Surprisingly, it was shown in \cite{Hec21} that the pure state spaces
$\mathcal{P}(C(\mathbb S^1)^{(n)})$, with metric induced by the Dirac operator of $\mathbb S^1$, converge to
$\mathcal{S}(C(\mathbb S^1))$ and not to $\mathcal{P}(C(\mathbb S^1))\simeq\mathbb S^1$ as one would have guessed.

Suppose, now, that $B$ is a C*-algebra and $A\subset B$ an operator system.
By Hahn-Banach theorem every state of $A$ can be extended to a state of $B$, and we get a surjective map
\begin{equation}\label{eq:F}
F:\mathcal{S}(B)\twoheadrightarrow\mathcal{S}(A)
\end{equation}
given by the restriction of states of $B$ to $A$.
The first two points in the next lemma,  that we prove just for the sake of completeness, appeared already in \cite[Fact 2.9]{CvS20}.

\begin{lemma}\label{lemma:previous}
Let $F$ be the map \eqref{eq:F} and $\varphi\in\mathcal{S}(A)$. Then:\vspace*{-3pt}
\begin{itemize}\itemsep=2pt
\item[(i)] The fiber $F^{-1}(\varphi)$ is a convex subset of $\mathcal{S}(B)$.
\item[(ii)] If $\varphi$ is pure, then extremal points of $F^{-1}(\varphi)$ are pure states of $B$.
\item[(iii)] If $F^{-1}(\varphi)=\{\psi\}$ is a singleton and $\psi\in\mathcal{P}(B)$, then $\varphi\in\mathcal{P}(A)$.
\end{itemize}
\end{lemma}

\begin{proof}
(i) Let $f,g\in F^{-1}(\varphi)$.
Then for any $\lambda\in\interval{0}{1}$ the combination
\begin{equation}\label{eq:decompose}
\psi:=\lambda f+(1-\lambda)g
\end{equation}
is still a state of $B$ and $F(\psi)=\lambda \varphi+(1-\lambda)\varphi=\varphi$, so that $\psi\in F^{-1}(\varphi)$.

(ii) Let now $\varphi$ be a pure state and $\psi$ an extremal point of $F^{-1}(\varphi)$. Suppose, by contradiction, that there exists $\lambda\in\ointerval{0}{1}$ and $f,g\in\mathcal{S}(B)$ such that $\psi$ can be decomposed like in \eqref{eq:decompose}. Then
$$
\varphi=F(\psi)=\lambda F(f)+(1-\lambda)F(g) .
$$
But $\varphi$ is pure, so it must be $F(f)=F(g)=\varphi$, that means $f,g\in F^{-1}(\varphi)$, contradicting the hypothesis that $\psi$ is an extremal point of $F^{-1}(\varphi)$.

(iii) Let $\varphi=\lambda\varphi_1+(1-\lambda)\varphi_2$ with $\lambda\in\interval{0}{1}$ and $\varphi_1,\varphi_2,\in\mathcal{S}(A)$.  Choose any $\psi_1\in F^{-1}(\varphi_1)$ and $\psi_2\in F^{-1}(\varphi_2)$ (recall that $F$ is surjective). Observe that
$$
\lambda\psi_1+(1-\lambda)\psi_2\in F^{-1}(\varphi) ,
$$
which implies $\psi=\lambda\psi_1+(1-\lambda)\psi_2$. But $\psi$ is pure, hence $\lambda=0$ or $\lambda=1$. This implies that $\varphi$ is pure as well.
\end{proof}

A natural question concerns injectivity of \eqref{eq:F}, or at least injectivity of its restriction to pure states. We will see a natural example of a convex-linear map whose restriction to pure states is injective in \S\ref{sec:6}, when discussing informationally complete positive operator valued measures (IC POVM).

In general, injectivity of a convex-linear map on extremal points does not guarantee injectivity on the full domain, as shown next.

\begin{ex}
Let $\Delta_3$ be the standard simplex
$$
\Delta_3:=\left\{p=(p_1,\ldots,p_4)\in\R^4:p_i\geq 0\;\forall\;i\text{ and }\sum p_i=1\right\} .
$$
Let $F:\Delta_3\to\interval{0}{1}^2$ be the map
$$
F(p):=(p_1+p_3,p_2+p_3) \;.
$$
It gives a bijection between the vertices of the tetrahedron and the vertices of the square. Since every convex set is the closed convex hull of its extremal points, $F(\Delta_3)\supset \interval{0}{1}^2$. The opposite inclusion is obvious, thus $F(\Delta_3)=\interval{0}{1}^2$. The map is clearly not a bijection, since for example $F\left(\frac{1}{2},\frac{1}{2},0,0\right)=F\left(0,0,\frac{1}{2},\frac{1}{2}\right)$.
\end{ex}

For the operator system of a tolerance relation on $\{1,\ldots,n\}$ it is easy to verify that the map \eqref{eq:F} is not injective, not even on pure states, unless $R$ is an equivalence relation.
(This could be derived from Corollary 3.9 in \cite{CvS21} where $L$ is the adjacency matrix of the graph of $R$, but we prefer to give an independent proof.)

\begin{prop}
Let $B:=M_n(\C)$, $R$ a tolerance relation on $\{1,\ldots,n\}$ and $A=A(R)$ its tolerance algebra.
The following are equivalent:
\begin{itemize}\itemsep=2pt
\item[(i)] the map \eqref{eq:F} is injective;
\item[(ii)] the restriction of \eqref{eq:F} to pure states is injective;
\item[(iii)] $A=B=M_n(\C)$.
\end{itemize}
\end{prop}

\begin{proof}
(iii) $\Rightarrow$ (i) $\Rightarrow$ (ii) is obvious.
We now prove that the negation of (iii) implies the negation of (ii).
Suppose that $R$ is not an equivalence relation, which means that there exists $x,y,z\in X$ such that $x\sim y$, $y\sim z$ and $x\not\sim z$.
We can assume $x=1,y=2,z=3$ (relabeling the elements of $X$ is equivalent to reordering the rows and columns of a matrix). For $u\in\C$ with $|u|\leq 1$, consider the state with density matrix:
$$
\rho_u=
\frac{1}{2}\left[\!\begin{array}{ccc|ccc}
1 & 0 & u & 0 & {\cdots} & 0 \\
0 & 0 & 0 & 0 & {\cdots} & 0 \\
u^* & 0 & 1 & 0 & {\cdots} & 0 \\
\hline
0 & 0 & 0 & 0 & {\cdots} & 0 \\[-2pt]
\vdots & \vdots & \vdots & \vdots & & \vdots \\
0 & 0 & 0 & 0 & {\cdots} & 0
\end{array}\!\right] \;.
$$
Since $\rho_u$ is a rank $1$ projection, the associated state $\varphi_u$, given by $\varphi_u(a)=\tr(\rho_u a)$ for all $a\in B$, is pure. Now we observe that, if $a\in A$:
$$
\varphi_u(a)=\tr(\rho_u T(a)) ,
$$
where $T$ is the truncation map in Example \ref{ex:tol12} and we used the fact that $T$ is the identity on $A$, since $T\circ T=T$. By cyclicity of the trace and symmetry of $R$:
\begin{equation}\label{eq:weuse}
\varphi_u(a)=\sum_{(i,j)\in R}\tr(\rho_u E_{ii}aE_{jj})=\sum_{(i,j)\in R}\tr(E_{ii}\rho_u E_{jj}a)=\tr(T(\rho_u)a) \;.
\end{equation}
But $T(\rho_u)$ is independent of $u$, since the truncation map kills the matrix elements in position $(1,3)$ and $(3,1)$. Thus, all states $\varphi_u$ have the same image under the map \eqref{eq:F}.
\end{proof}

\subsection{Positivity in a tolerance algebra}

As in the previous section, let $R$ be a tolerance relation on the set $\{1,\ldots,n\}$, $A=A(R)\subset M_n(\C)$ the vector space in Example \ref{ex:tol12}, $T$ the map \eqref{eq:trunmapfin}, $\star$ the product \eqref{eq:matrix}. 
There are two natural partial orders on $A(R)$, which we will denote by $\geq$ and $\succeq$, defined as follows. For $a\in A(R)$ we will write
$a\geq 0$ if $a$ is a positive semidefinite matrix, and $a\succeq 0$ if there exists $b\in M_n(\C)$ such that $b\geq 0$ and $a=T(b)$.

Observe that, since $T$ is idempotent, this in particular means that, for all $a\in A$, $a\geq 0$ implies $a\succeq 0$. On the other hand, for every element $a\in A$ of the form
\begin{equation}\label{eq:naturalnotion}
a=\sum_{i=1}^kb_i\star b_i^* ,
\end{equation}
with $b_1,\ldots,b_k\in A(R)$, one has $a=T(\sum_{i=1}^kb_ib_i^*)\succeq 0$ even if in general \eqref{eq:naturalnotion} may be not a positive semidefinite matrix (see Example \ref{ex:tol12}).

Elements \eqref{eq:naturalnotion} are the natural candidates for positive elements in a $*$-algebra. The relation with the partial order $\succeq$ is illustrated in the next proposition.

Recall that in an undirected graph $\Gamma=(V,E)$, a \emph{dominant vertex} is a vertex that is adjacent to all other vertices of the graph.

\begin{prop}\label{prop:counter}
Assume that in each connected component of the graph of $R$ there is a dominant vertex. Then, for all $a\in A(R)$, one has $a\succeq 0$ if and only if $a$ is of the form \eqref{eq:naturalnotion} for some $b_1,\ldots,b_k\in A(R)$.
\end{prop}

\begin{proof}
We have to prove the implication $\Rightarrow$, the other being always true. Clearly every $a\succeq 0$ is of the form $T(\sum_{i=1}^kb_ib_i^*)$ for some $b_1,\ldots,b_k\in M_n(\C)$ (in fact, even with $k=1$). The non-trivial part is to show that one can choose $b_1,\ldots,b_k$ belonging to $A(R)$.

Up to a permutation of rows and columns, which preserves both the partial order $\succeq$ and the decompositions of the form \eqref{eq:naturalnotion}, we can think of $A(R)\subset M_n(\C)$ as subset of block diagonal matrices where each block corresponds to a connected component of the graph of $R$.
It is then enough to prove the proposition for each block, i.e.~under the assumption that the graph of $R$ is connected.

By hypothesis, there exists $j_0$ such that $(i,j_0)\in R$ for all $i\in\{1,\ldots,n\}$.
Let $a=T(c)$ with $c\in M_n(\C)$ positive semidefinite. Since $T$ is linear and every positive semidefinite matrix is a linear combination with positive coefficients of rank $1$ projections, it is enough to prove the statement when $c$ is itself a rank $1$ projection. Thus, $c=\sum_{i,j=1}^nv_iv_j^*E_{ij}$ for some unit vector $v=(v_1,\ldots,v_n)\in\C^n$. Let $b:=\sum_{i=1}^nv_iE_{ij_0}$ and observe that $b\in A(R)$. Since
$bb^*=c$, this concludes the proof.
\end{proof}

We will see that the assumption in Prop.~\ref{prop:counter} is necessary (cf.~Prop.~\ref{prop:isnecessary}).

As recalled in the previous section, we use the partial order $\geq$ to define $\mathcal{S}(A)$. Explicitly, we define a partial order $\geq$ on the dual vector space $A^*$ by setting $\varphi\geq 0$ if and only if $\varphi(a)\geq 0$ for all $a\geq 0$. A state $\varphi\in\mathcal{S}(A)$ will then be a positive element of $A^*$ with normalization $\varphi(1)=1$.
With such a definition, the restriction of a state of $M_n(\C)$ to $A$ is a state, and \eqref{eq:F} is well defined.

Even if we use $\geq$ to define states, when trying to describe states in terms of matrices the other partial order $\succeq$ pops up, as explained below. This stems from the next crucial observation.

\begin{prop}\label{prop:orderiso}
For $x\in M_n(\C)$, let $\varphi_x:M_n(\C)\to \C$ be the linear map defined by
$$
\varphi_x(a):=\tr(x^*a)\qquad\forall\;a\in M_n(\C).
$$
Then, the map $(A,\succeq)\to (A^*,\geq)$ sending $x\in A$ to $\varphi_x|_A$ is an isomorphism of ordered vector spaces.
\end{prop}

Recall that an isomorphism of ordered complex vector spaces is an isomorphism of the underlying real vector spaces that is compatible with the partial orders (this remark is necessary since $x\mapsto\varphi_x$ is an antilinear map), cf.~\cite{SW99}.
Notice that the isomorphism above transforms the partial order $\succeq$, defined using the non-associative product, into the partial order $\geq$, which only depends on the underlying operator system.

\begin{proof}
For $a,b\in M_n(\C)$ let $\inner{a,b}:=\tr(a^*b)$ be the Hilbert-Schmidt inner product. Its restriction to $A$ is still an inner product, and from its non-degeneracy it follows that the map $A\to A^*$, $x\mapsto\varphi_x=\inner{x,\,.\,}$, is injective. Since the vector spaces $A$ and $A^*$ have the same dimension, such a map is surjective as well, hence an isomorphism of real vector spaces. It remains to show that,
for all $x\in A$, one has
\begin{equation}
x\succeq 0\iff \varphi_x\geq 0.
\end{equation}
By Hahn-Banach theorem, every positive element in $A^*$ is the restriction of a positive linear functional $M_n(\C)\to\C$.
Thus $\varphi_x\geq 0$ if and only if there exists $b\in M_n(\C)$ such that $b\geq 0$ and
$$
\varphi_x(a)=\varphi_b(a)=\tr(ba) \qquad\forall\;a\in A.
$$
Repeating the proof of \eqref{eq:weuse} one finds that
$$
\varphi_b(a)=\tr(T(b)a) \qquad\forall\;a\in A \;,
$$
thus $\inner{x-T(b),a}=0$ for all $a\in A$, which implies $x=T(b)$. Since $b\geq 0$, it follows that $x\succeq 0$.
This proves ``$\Leftarrow$''.

Conversely, if $x\succeq 0$, then $x=T(b)$ for some positive $b\in M_n(\C)$ and
$\varphi_x=\varphi_{T(b)}|_A=\varphi_b|_A$ is the restriction to $A$ of the positive linear functional \mbox{$\varphi_b:M_n(\C)\to\C$}, hence it is positive itself.
\end{proof}

\begin{df}
Let us denote by $\mathcal{D}(A)$ the set of all $\rho\in A$ such that $\rho\succeq 0$ and $\tr(\rho)=1$.
\end{df}

If $A=M_n(\C)$ (thus, $R$ is the trivial relation $i\sim j$ for all $i,j=\{1,\ldots,n\}$), the one above is the standard definition of \emph{density matrix}.

In general, since the relation $\geq$ is contained in the relation $\succeq$, every positive semidefinite matrix with trace $1$ (every density matrix) belongs to $\mathcal{D}(A)$, but there may be matrices in $\mathcal{D}(A)$ that are not positive semidefinite.
We may call an element $\rho\in A$ with trace $1$ a \emph{strict} density matrix if $\rho\geq 0$, and a \emph{weak} density matrix if $\rho\succeq 0$.\footnote{Observe that, in the physics language, a strict density matrix \emph{is} a density matrix, while a weak density matrix is not.}

\begin{ex}
Let $R$ be the relation \eqref{eq:exgraph} and
$$
\rho=\frac{1}{3}\begin{pmatrix}
1 & 1 & 0 \\ 1 & 1 & 1 \\ 0 & 1 & 1
\end{pmatrix}=\frac{1}{3}\;T\!\begin{pmatrix}
1 & 1 & 1 \\ 1 & 1 & 1 \\ 1 & 1 & 1
\end{pmatrix} .
$$
This is a weak density matrix for $A(R)$, but not a strict one since $\det(\rho)<0$.
\end{ex}

An immediate consequence of Prop.~\ref{prop:orderiso} is that states on $A$ in the sense of operator systems are in bijection with weak density matrices in the sense of (non-associative) tolerance algebras. The same conclusion could be reached using Prop.~3.10 of \cite{CvS21}.

\begin{cor}
The map $\mathcal{D}(A)\to\mathcal{S}(A)$, $\rho\mapsto\varphi_\rho$, is a bijection (in fact, it is a homeomorphism if we put the norm topology on $\mathcal{D}(A)$ and the weak-$^*$ topology on $\mathcal{S}(A)$).
\end{cor}

\subsection{Pure states of a tolerance algebra}
Let $R$ be a tolerance relation on the set $\{1,\ldots,n\}$, $A(R)$ and $T$ as in the previous section.
Here we want to give an explicit description of pure states of $A(R)$.

Given a non-zero vector $v=(v_1,\ldots,v_n)\in\C^n$, then
\begin{equation}\label{eq:Rv}
R_v :=\big\{ (i,j)\in R: v_iv_j\neq 0 \big\}
\end{equation}
is a tolerance relation on the (non-empty) subset of $i\in\{1,\ldots,n\}$ such that $v_i\neq 0$. In terms of graphs, $R_v$ is the subgraph of $R$ obtained by removing all vertices $i$ such that $v_i=0$, and all edges that are incident on such vertices.

\begin{df}\label{df:25}
A non-zero vector $v\in\C^n$ is called $R$-\emph{tolerant} if the graph of $R_v$ is connected.
A pure state of $M_n(\C)$ is called $R$-\emph{tolerant} if its density matrix is a projection in the direction of an $R$-tolerant vector.
\end{df}

Notice that in the previous definition we do not assume that the graph of $R$ is connected.

Pictorially, some examples are illustrated in Figure \ref{fig:three}, where vertices are labelled by the components of $v$.
In the first example (on the left), since $v_3$ is adjacent to every other vertex, if $v_3\neq 0$ the graph of $R_v$ is connected; if $v_3=0$ and $v_4\neq 0$, $R_v$ is disconnected; if $v_3=v_4=0$, then $R_v$ is connected. Similarly in the third example (on the right), $R_v$ is disconnected if and only if $v_2=0$ and both $v_1$ and $v_3$ are non-zero. In the second example (in the middle), we see the opposite phenomenon: if both $v_1$ and $v_2$ are non-zero, $R_v=R$ is disconnected; if $v_1=0$ or $v_2=0$, $R_v$ is connected.

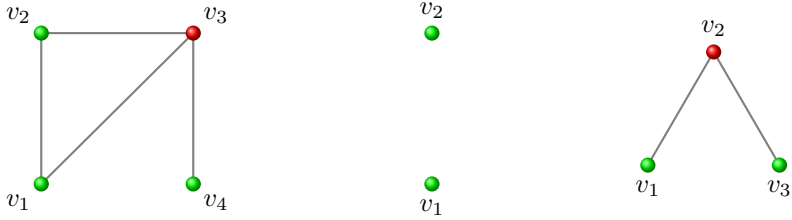
\begin{figure}[t]
\begin{tikzpicture}[font=\small,baseline=(current bounding box.center)]

\draw[gray,thick] (2,0) -- (2,2) -- (0,2) -- (0,0) -- (2,2);

\shade[ball color=green] (0,0) circle (0.1) node[below left] {$v_1$};
\shade[ball color=green] (0,2) circle(0.1) node[above left] {$v_2$};
\shade[ball color=red] (2,2) circle(0.1) node[above right] {$v_3$};
\shade[ball color=green] (2,0) circle(0.1) node[below right] {$v_4$};

\end{tikzpicture}
\hspace{2cm}
\begin{tikzpicture}[font=\small,baseline=(current bounding box.center)]

\shade[ball color=green] (0,2) circle(0.1) node[above=2pt] {$v_2$};
\shade[ball color=green] (0,0) circle (0.1) node[below=2pt] {$v_1$};

\end{tikzpicture}
\hspace{2cm}
\begin{tikzpicture}[font=\small,baseline=(current bounding box.center)]

\draw[gray,thick] (330:1) -- (90:1) -- (210:1);

\shade[ball color=red] (90:1) circle(0.1) node[above=2pt] {$v_2$};
\shade[ball color=green] (210:1) circle (0.1) node[below=2pt] {$v_1$};
\shade[ball color=green] (330:1) circle (0.1) node[below=2pt] {$v_3$};

\end{tikzpicture}

\caption{Some tolerance relations discussed in the text.}\label{fig:three}
\end{figure}

Inspired by these ``experiments'', let us record some more general examples.

\begin{ex}
If $A(R)=M_n(\C)$, then every non-zero vector $v\in\C^n$ is $R$-tolerant ($R_v$ is connected for all $v\neq 0$, since a complete graph cannot be disconnected by deleting vertices).

If $A(R)=\C^n$, then $R$-tolerant vectors are those with only one non-zero component (corresponding to diagonal projection matrices).

If $R$ is the relation \eqref{eq:exgraph}, then $v\in\C^3$ is $R$-tolerant except when $v_2=0$ and $v_1,v_3\neq 0$.
\end{ex}

Under the isomorphism between a vector subspace of $M_n(\C)$ and its dual, given by the Hilbert-Schmidt inner product, the map
$F:\mathcal{S}(M_n(\C))\to \mathcal{S}(A(R))$ in \eqref{eq:F} becomes the restriction to density matrices of the truncation map $T:M_n(\C)\to A(R)$.

If $v\in\C^n$ is a unit vector, let us denote by $P_v:=(v_iv_j^*)\in M_n(\C)$ the corresponding projection.
Because of Lemma \ref{lemma:previous}(ii), every pure state of $A(R)$ is the restriction of a pure state of $M_n(\C)$, hence it has weak density matrix
$T(P_v)$ for some unit vector $v\in\C^n$. The remaining part of this section is devoted to the proof that:

\begin{prop}\label{prop:Rtol}
$T(P_v)$ is the weak density matrix of a pure state of $A(R)$ if and only if $v\in\C^n$ is an $R$-tolerant unit vector. More precisely, $F$ gives a bijection between the subset of pure states of $M_n(\C)$ that are $R$-tolerant and pure states of $A(R)$.
\end{prop}

Before we prove the proposition, let us note that the definition of $R$-tolerant vector is very close to the one of
state with $\epsilon$-connected support in \cite[Sect.~4.2]{CvS21}. However, the latter concerns proximity relations on a metric space, while our more general Definition \ref{df:25} is for arbitrary tolerance relations on a finite set. There is, of course, a non-empty intersection between the two cases. On this intersection, Prop.~\ref{prop:Rtol} is analogous to Theorem 4.18 \cite{CvS21}.

We need a few preliminary lemmas.

\begin{lemma}\label{lemma:descdiag}
Let $k\geq 3$, let $w\in\C^k$ be a vector with $w_i\neq 0$ for all $i\in\{1,\ldots,k\}$, and let $a=(a_{ij})\in M_k(\C)$ be a positive semidefinite matrix such that $a_{ij}=w_iw_j^*$ for all $|i-j|\leq 1$. Then $a_{ij}=w_iw_j^*$ for all $i,j\in\{1,\ldots,k\}$.
\end{lemma}

\begin{proof}
We are claiming that $a$ is uniquely determined by its main diagonal and the two adjacent descending diagonals. We will prove the statement by induction on $k$.
Let then $a\in M_k(\C)$ be a matrix with elements $a_{ij}=w_iw_j^*$ if $|i-j|\leq 1$ and call $\lambda_{ij}:=a_{ij}/w_iw_j^*$ if $|i-j|>1$.
If $k=3$, one has
$$
a=\begin{pmatrix}
|w_1|^2 & w_1w_2^* & \lambda_{13}w_1w_3^* \\
w_2w_1^* & |w_2|^2 & w_2w_3^* \\
\lambda_{13}^*w_3w_1^* & w_3w_2^* & |w_3|^2
\end{pmatrix} .
$$
Since
$$
\det(a)=(\lambda_{13}+\lambda_{13}^*-|\lambda_{13}|^2-1)|w_1w_2w_3|^2
$$
and $w_1w_2w_3\neq 0$, the condition $\det(a)\geq 0$ implies $\lambda_{13}+\lambda_{13}^*-|\lambda_{13}|^2-1\geq 0$. The unique solution of such an inequality is $\lambda_{13}=1$, which is what we wished to prove.

Now let $k\geq 4$ be arbitrary and assume by inductive hypothesis that the lemma is true for matrices of size $k-1$. If $a=(a_{ij})\in M_k(\C)$ is positive semidefinite and $a_{ij}=w_iw_j^*$ for all $|i-j|\leq 1$,
the matrix $a'=(a_{ij})\in M_{k-1}(\C)$ obtained from $a$ by removing the last row and column is positive semidefinite and satisfies $a'_{ij}=w_iw_j^*$ for all $|i-j|\leq 1$. By inductive hypothesis $a'_{ij}=w_iw_j^*$ for all $i,j\in\{1,\ldots,k-1\}$.
Thus, the principal minor of $a$ given by elements in the intersection of the rows and columns $\{1,k-1,k\}$ is:
$$
\det\begin{pmatrix}
|w_1|^2 & w_1w_{k-1}^* & \lambda_{1k}w_1w_k^* \\
w_{k-1}w_1^* & |w_{k-1}|^2 & w_{k-1}w_k^* \\
\lambda_{1k}^*w_kw_1^* & w_kw_{k-1}^* & |w_k|^2
\end{pmatrix}=(\lambda_{1k}+\lambda_{1k}^*-|\lambda_{1k}|^2-1)|w_1w_{k-1}w_k|^2 .
$$
The determinant above is $\geq 0$ if and only if $\lambda_{1k}=1$, which completes the inductive step.
\end{proof}

\begin{lemma}\label{lemma:tolsing}
If $\rho\in M_n(\C)$ is a density matrix, $v\in\C^n$ an $R$-tolerant unit vector and $T(\rho)=T(P_v)$, then $\rho=P_v$.
\end{lemma}

\begin{proof}
Let $v\in\C^n$ be an $R$-tolerant unit vector and $\rho=(\rho_{ij})\in M_n(\C)$.
The condition $T(\rho)=T(P_v)$ gives $\rho_{ij}=v_iv_j^*$ for all $(i,j)\in R$.
Using the condition $\rho\geq 0$ we now prove that $\rho_{ij}=v_iv_j^*$ for all $(i,j)\notin R$ as well. Since we are working with Hermitian matrices, it is enough to consider $i<j$.
Thus, in the rest of the proof $(i,j)\notin R$ and $i<j$, and we assume that $\rho\geq 0$.

If $v_i=0$, since the submatrix:
$$
\begin{pmatrix}
\rho_{ii} & \rho_{ij} \\
\rho_{ji} & \rho_{jj}
\end{pmatrix}=
\begin{pmatrix}
0 & \rho_{ij} \\
\rho_{ij}^* & |v_j|^2
\end{pmatrix}
$$
must have non-negative determinant, we deduce that $\rho_{ij}=0$. But $v_iv_j^*=0$ as well, so $\rho_{ij}=v_iv_j^*$.
The case $v_j=0$ is analogous.

If both $v_i$ and $v_j$ are non-zero,
since $v$ is $R$-tolerant, there exists a path in $R_v$ from $v_i$ to $v_j$. That is, there exists non-zero components $v_{l_1},\ldots,v_{l_k}$ such that 
$l_1=i$, $l_k=j$ and $(l_r,l_{r+1})\in R$ for all $1\leq r<k$. Since $(i,j)\notin R$, it must be $k\geq 3$.

For $r,s\in\{1,\ldots,k\}$, call $a_{rs}:=\rho_{l_rl_s}$ and $w_r:=v_{l_r}$. Observe that $a=(a_{rs})\in M_k(\C)$ and $w=(w_r)\in\C^k$
satisfy the hypotheses of Lemma \ref{lemma:descdiag}. We conclude that $a_{rs}=w_rw_s^*$ for all $r,s\in\{1,\ldots,k\}$.
In particular, for $r=1$ and $s=k$ we get $\rho_{ij}=a_{1k}=w_1w_k^*=v_iv_j^*$, that is exactly what we wished to prove.
\end{proof}

\begin{lemma}\label{lemma:pureconnected}
Let $w=(w_1,\ldots,w_k)\in\C^k$ be a unit vector with $w_i\neq 0$ for all $i\in\{1,\ldots,k\}$, $R'$ a tolerance relation on $\{1,\ldots,k\}$ and $T:M_k(\C)\to A(R')$ the corresponding truncation map. If $T'(P_w)$ is the weak density matrix of a pure state of $A(R')$, then the graph of $R'$ must be connected.
\end{lemma}

\begin{proof}
By contradiction, assume that the graph of $R'$ is not connected. After renaming the vertices and reordering rows and columns we can assume that, for some $m\in\{1,\ldots,k-1\}$, the first $m$ vertices are disconnected from the last $k-m$, that is:
if $i\leq m<j$ then $(i,j)\notin R'$.
This means that $T(P_w)$ is a block diagonal matrix of the form:
$$
T(P_w)=\left(\!\begin{array}{c|c}
Q' \\
\hline
& Q''
\end{array}\!\right)
$$
where $Q'\in M_m(\C)$, $Q''\in M_{k-m}(\C)$, and off-diagonal blocks are zero. Let
$$
t:=\sum_{i=1}^m|w_i|^2
$$
and observe that, since $w$ has unit norm, $1-t=\sum_{i=m+1}^k|w_i|^2$. Since $w$ has all components different from zero, both $t$ and $1-t$ are non-zero, that means $0<t<1$.
Consider the unit vectors:
$$
w':=t^{-1/2}(w_1,\ldots,w_m,0,\ldots,0) \qquad\text{and}\qquad
w'':=(1-t)^{-1/2}(0,\ldots,0,w_{m+1},\ldots,w_k) .
$$
Clearly
$$
\left(\!\begin{array}{c|c}
Q' \\
\hline
& 0
\end{array}\!\right)=t\, T(P_{w'})\qquad\text{and}\qquad
\left(\!\begin{array}{c|c}
0 \\
\hline
& Q''
\end{array}\!\right)=(1-t)\, T(P_{w''}) .
$$
Thus $T(P_w)=t\, T(P_{w'})+(1-t)\, T(P_{w''})$, and the corresponding state of $A(R')$ is not pure.
\end{proof}

We are now ready to prove Proposition \ref{prop:Rtol}.

\begin{proof}[Proof of Prop.~\ref{prop:Rtol}]
Let $v\in\C^n$ be a unit vector and denote by $\varphi$ the state of $A(R)$ with weak density matrix $T(P_v)$.

If $v$ is $R$-tolerant, it follows from Lemma \ref{lemma:tolsing} that $F^{-1}(\varphi)=\{\psi\}$ is a singleton, with $\psi$ the state with density matrix $P_v$.
Thus $F$ is injective on $R$-tolerant pure states of $M_n(\C)$.
From Lemma \ref{lemma:previous}(iii), it also follows that $\varphi$ is a pure state of $A(R)$.
It remains to prove that if $v$ is not $R$-tolerant, then $\varphi$ is not a pure state of $A(R)$.

Without loss of generality, we can assume that the non-zero components of $v$ are the first $k$ ($k\geq 1$).
We can identify $A(R_v)$ (risp.~$M_k(\C)$) with the vector subspace of $A(R)$ (risp.~$M_n(\C)$) of matrices with the last $n-k$ rows and columns filled with zeros. 
We now apply Lemma \ref{lemma:pureconnected} to $R'=R_v$ and $w=(v_1,\ldots,v_k)$. Since $R_v$ is not connected,
$T'(P_w)$ is not the weak density matrix of a pure state of $A(R')$. That means there are two distinct density matrices $\rho_1,\rho_2\in M_k(\C)$ and a $t\in\ointerval{0}{1}$ such that $T'(P_w)=tT'(\rho_1)+(1-t)T'(\rho_2)$. But $T(a)=T'(a)$ for all $a\in M_k(\C)$, hence $T(P_v)=T'(P_w)=tT(\rho_1)+(1-t)T(\rho_2)$ is a non-trivial convex combination of two weak density matrices. We conclude that the state $\varphi$ is not pure.
\end{proof}

As a corollary of the above results, we now show that the assumption in Prop.~\ref{prop:counter} is necessary.

\begin{prop}\label{prop:isnecessary}
Let $R$ be a tolerance relation on $X:=\{1,\ldots,n\}$ and assume that each $a\succeq 0$ in $A(R)$ is of the form \eqref{eq:naturalnotion}, for some $b_1,\ldots,b_k\in A(R)$. Then the graph of $R$ has a dominant vertex in each connected component.
\end{prop}

\begin{proof}
Similarly to the proof of Prop.~\ref{prop:counter}, it is enough to prove the statement when the graph of $R$ is connected.
Since $R$ is connected, $w:=(1,1,\ldots,1)$ is an $R$-tolerant vector.
Called $P:=n^{-1}(w_iw_j^*)$, we now prove that $a:=T(P)$ can be written in the form \eqref{eq:naturalnotion} only if there the graph of $R$ has a dominant vertex (at least one).

Assume that $a=T(\sum_{i=1}^kb_ib_i^*)$ for some $b_1,\ldots,b_k\in A(R)$. From Lemma \ref{lemma:tolsing} it follows
\begin{equation}\label{eq:becauseof}
{\textstyle\sum_{i=1}^kb_ib_i^*}=P .
\end{equation}
From this we deduce that
$$
\|b_i^*v\|^2=\inner{v,b_ib_i^*v}\leq\inner{v,Pv}=0
$$
--- i.e.~$b_iv=0$ --- for all $i=1,\ldots,k$ and for every $v$ in the kernel of $P$.
Because of \eqref{eq:becauseof}, not all matrices $b_ib_i^*$'s can be zero: for some $i_0\in\{1,\ldots,k\}$, one has $b_{i_0}b_{i_0}^*\neq 0$.
Since $b_{i_0}b_{i_0}^*$ vanishes on the kernel of the rank $1$-projection $P$, it must be $b_{i_0}b_{i_0}^*=\lambda P$ for some $\lambda >0$. Called $b:=\lambda^{-1/2}b_{i_0}$, we proved that
$$
P=bb^*
$$
for some $b\in A(R)$. Since $\ker P=\ker b^*$, the matrix $b^*$ must be of rank $1$, hence $b^*=n^{-1/2}(u_iw_j^*)$ for some unit vector $u\in\C^n$, and with $w$ the vector defined above. That is:
$$
\sqrt{n}\,b^*=\begin{pmatrix}
u_1 & u_1 & \ldots & u_1 \\
u_2 & u_2 & \ldots & u_2 \\
\vdots & \vdots & & \vdots \\
u_n & u_n & \ldots & u_n
\end{pmatrix} .
$$
At least one row of $b^*$, say $j_0$, must be non-zero. Since $b^*\in A(R)$ and in the row $j_0$ all its elements are different from zero, this means that $(i,j_0)\in R$ for all $i=1,\ldots,n$. Thus, $j_0$ is a dominant vertex in the graph of $R$.
\end{proof}

\section{Positive operator valued measures: a curious example}\label{sec:6}

Given an observable, represented by a Hermitian matrix $P$, and a physical system in a state described by a density matrix $\rho$ of the same size of $P$, a measurement of such observable returns an eigenvector of $P$.
Order the eigenvalues and label them with an integer, say from $1$ to $k$.
The outcome of a measurement is then an integer $i\in\{1,\ldots,k\}$, telling us which eigenvalue comes out of our measurement apparatus.
Performing the experiment many times gives the frequency of the outcome $i$, that is an estimate of the probability $\rho^P(i)$ that the state collapses in an eigenstate associated to the eigenvalue $i$. Such a probability is expressed by the formula:
\begin{equation}\label{eq:rhotorhoP}
\rho^P(i):=\tr(\rho P_i) ,
\end{equation}
where $P_i$ is the spectral projection of the $i$-th eigenvalue of $P$.
The spectral projections are an example of a positive operator valued measure (POVM), which in the simplest case are a collection
of $k$ positive semidefinite $n\times n$ matrices $P_i$, with $i=1,\ldots,k\leq n$, such that $\sum_{i=1}^kP_i=1$.
The complex linear span of the $P_i$'s is a concrete operator system $A\subset M_n(\C)$.
From now on, we will denote by $P$ a collection of positive semidefinite $n\times n$ matrices $P_1,\ldots,P_k$.

POVMs arise in the study of physical systems of which we have partial knowledge (``open'' quantum systems), and an excellent textbook reference is \cite{Per02}.

Given a POVM, $P:=(P_i)_{i=1}^k$, the formula \eqref{eq:rhotorhoP} defines a convex-linear map
\begin{equation}\label{eq:POVMcond}
\rho\mapsto\rho^P
\end{equation}
from density matrices to probability distributions on $k$ points, and a natural question is whether this map is injective: i.e.~the knowledge of the probabilities $\rho^P(i)$ allows one to reconstruct the state $\rho$. If the map \eqref{eq:POVMcond} is injective, the POVM $P$ is called ``informationally complete'' (IC). Since, in such a case, the map gives a smooth embedding of a $n^2-1$ smooth manifold (the interior of the set of $n\times n$ density matrices) into the interior of a $k-1$ simplex, this forces $k\geq n^2$.

Let us consider on $M_n(\C)$ the inner product $\inner{a,b}:=\tr(a^*b)$, and extend \eqref{eq:POVMcond} to a linear map $L:M_n(\C)\to \R^k$ in the obvious way:
\begin{equation}\label{eq:mapL}
L(a):=\left(\inner{P_1,a},\ldots,\inner{P_k,a}\right).
\end{equation}

\begin{lemma}\label{lemma:uniqueness}
Let $P=(P_i)_{i=1}^k$ be a POVM. Then, the convex-linear map \eqref{eq:POVMcond} is injective if and only if the linear map \eqref{eq:mapL} is.
\end{lemma}

\begin{proof}
The implication ``$\Leftarrow$'' is clear. We now prove ``$\Rightarrow$''.
Every $a\in\ker L$ has $\tr(a)=\sum_{i=1}^k\inner{a,P_i}=0$. Assume that $\ker L\neq\{0\}$. Then we can find a non-zero Hermitian matrix $a$ in the kernel, that can be written as $a=b_1-b_2$ with $b_1$ and $b_2$ two distinct positive semidefinite matrices with $\tr(b_1)=\tr(b_2)=:\lambda$. The trace of a positive semidefinite matrix is zero if and only if the matrix is zero, but $b_1$ and $b_2$ cannot be both zero, thus $\lambda\neq 0$. Hence $\rho_1:=\lambda^{-1}b_1$ and $\rho_2:=\lambda^{-1}b_2$ are two distinct density matrices such that $\rho_1^P=\rho_2^P$, and \eqref{eq:POVMcond} is not injective.
\end{proof}

If $A$ is the operator system associated to our POVM, since $\ker L=A^\perp$, an immediate corollary of the previous lemma is the well-known fact that the POVM is IC if and only if $A=M_n(\C)$.

We now change setting and pass to diagonal matrices, that is we consider a finite \emph{classical} system. 
Notice that, in fact, Lemma \ref{lemma:uniqueness} holds for any C*-subalgebra of $M_n(\C)$ (direct sum of matrix algebra), and in particular for diagonal matrices, i.e.~finite classical systems.

\begin{wrapfigure}[7]{r}{0.2\textwidth}
\begin{center}
\begin{tikzpicture}

\draw[gray!50] (0,0) circle (1);

\foreach \x in {0,30,60}{
          \shade[ball color=yellow] (\x:1) circle(0.1){};
      }

\foreach \x in {90,120,150,180,210,240,270,300,330}{
          \fill[gray!50] (\x:1) circle(0.1);
      }

\end{tikzpicture}
\end{center}

\vspace{-6pt}

\setcaptionwidth{2cm}%
\caption{}\label{fig:1}
\end{wrapfigure}
Let us consider the following simple classical example.
Fix two integers $1\leq k\leq n$.
Suppose we have $n$ points on a circle and a detector with $n$ LEDs at those points (Figure \ref{fig:1}). Only groups of $k$ consecutive LEDs can turn on simultaneously, and when a group is on we know that a particle passed throught one of those $k$ points, but not which one.
The detector cannot distinguish between $k$ consecutive points, meaning that we have a tolerance relation:
$$
i\sim j\iff |i-j|<k .
$$
This is clearly not transitive if $k\notin\{1,n\}$.

The POVM describing this detector is $P=(\tfrac{1}{k}Q_i)_{i=1}^n$, with $Q_i\in M_n(\C)$ the diagonal matrix with $1$ in positions $i,i+1,\ldots,i+k-1$ (mod $n$) and zero everywhere else. For example, for $k=3$ and $n=5$ one has:
$$
Q_1=\left[\begin{smallmatrix}
1 & & & & \\
 & 1 & & & \\
 & & 1 & & \\
 & & & 0 & \\
 & & & & 0
\end{smallmatrix}\right] \hspace{8pt}
Q_2=\left[\begin{smallmatrix}
0 & & & & \\
 & 1 & & & \\
 & & 1 & & \\
 & & & 1 & \\
 & & & & 0
\end{smallmatrix}\right] \hspace{8pt}
Q_3=\left[\begin{smallmatrix}
0 & & & & \\
 & 0 & & & \\
 & & 1 & & \\
 & & & 1 & \\
 & & & & 1
\end{smallmatrix}\right] \hspace{8pt}
Q_4=\left[\begin{smallmatrix}
1 & & & & \\
 & 0 & & & \\
 & & 0 & & \\
 & & & 1 & \\
 & & & & 1
\end{smallmatrix}\right] \hspace{8pt}
Q_5=\left[\begin{smallmatrix}
1 & & & & \\
 & 1 & & & \\
 & & 0 & & \\
 & & & 0 & \\
 & & & & 1
\end{smallmatrix}\right] \!,
$$
where the zeros outside the main diagonal are omitted. If
$$
\rho=\mathrm{diag}( \lambda_1 , \lambda_2 , \ldots , \lambda_n )
$$
is a density matrix (a probability distribution on $n$ points), then
$$
\rho^P(i)=\frac{1}{k}\sum_{j=i}^{i+k-1}\lambda_j
$$
where the index in the sum is defined mod $n$. Since the case $k=1$ is trivial, from now on let us assume that $k\geq 2$.

\begin{prop}
The POVM $P=(\tfrac{1}{k}Q_i)_{i=1}^n$ is IC if and only if $k$ and $n$ are coprime.
\end{prop}

\begin{proof}
It follows from Lemma \ref{lemma:uniqueness} that \eqref{eq:POVMcond} is injective if and only if
the map $L:\C^n\to\C^n$, sending $(\lambda_1,\ldots,\lambda_n)$ to the tuple with elements $\frac{1}{k}\sum_{j=i}^{i+k-1}\lambda_j$, $1\leq i\leq n$, is injective. The representative matrix of $L$ is the $n\times n$ matrix:
$$
M:=\begin{blockarray}{cccccccccccc}
\BAmulticolumn{6}{c}{\overbrace{\hspace*{3.1cm}}^{k\text{ times}}} & \\[-5pt]
\begin{block}{[cccccccccccc]}
1 & 1 & 1 & \ldots & 1 & 1 & 0 & 0 & 0 & \ldots & 0 & 0 \\
0 & 1 & 1 & \ldots & 1 & 1 & 1 & 0 & 0 & \ldots & 0 & 0 \\
0 & 0 & 1 & \ldots & 1 & 1 & 1 & 1 & 0 & \ldots & 0 & 0 \\
\vdots & \vdots & \vdots && \vdots & \vdots & \vdots & \vdots & \vdots && \vdots & \vdots \\
1 & 1 & 1 & \ldots & 1 & 0 & 0 & 0 & 0 & \ldots & 0 & 1 \\
\end{block}
\end{blockarray}
=S^0+S^1+S^2+\ldots+S^{k-1} ,
$$
where here and in the following we denote by $C$ the \emph{clock} and by $S$ the \emph{shift}, given by
\begingroup
\setlength{\arraycolsep}{4pt}%
\renewcommand{\arraystretch}{0.85}%
$$
C:=\begin{bmatrix}
q & 0 & 0 & \ldots & 0 \\
0 & q^2 & 0 & \ldots & 0 \\
0 & 0 & q^3 & \ldots & 0 \\[-2pt]
\vdots & \vdots & \vdots & & \vdots \\
0 & 0 & 0 & \ldots &  q^n
\end{bmatrix} \qquad\quad
S:=\begin{bmatrix}
0 & 1 & 0 & 0 & \ldots & 0 \\
0 & 0 & 1 & 0 & \ldots & 0 \\
0 & 0 & 0 & 1 & \ldots & 0 \\[-2pt]
\vdots & \vdots & \vdots & \vdots && \vdots \\
0 & 0 & 0 & 0 & \ldots & 1 \\
1 & 0 & 0 & 0 & \ldots & 0
\end{bmatrix} ,
$$
\endgroup
and $q$ is a primitive $n$-th root of $1$. The shift can be transformed into the clock by a unitary transformation, so that $M$ is conjugated to the matrix $C^0+C^1+\ldots+C^{k-1}$. Such a matrix is diagonal, with element in position $(i,i)$ given by
\begin{equation}\label{eq:zero}
\sum_{j=0}^{k-1}q^{ij}
\end{equation}
for all $i=1,\ldots,n$.
The coefficient matrix is invertible if and only if \eqref{eq:zero} is different from zero for all $i=1,\ldots,n$.

Observe that, if $i=n$, then $q^i=1$ and \eqref{eq:zero} equals $k$. If $i<n$, since $q$ is a primitive $n$-th root of $1$ we have $q^i\neq 1$ and
$$
\sum_{j=0}^{k-1}q^{ij}=\frac{1-q^{ik}}{1-q^i}.
$$
Thus, \eqref{eq:zero} is zero if and only if $q^i$ is a $k$-th root of unity, that is $ik$ is a multiple of $n$.
Hence, $L$ is injective if and only if there is no integer $ik<nk$ that is a multiple of $n$, i.e.~there is no integer less than $nk$ that is a common multiple of $n$ and $k$. But this exactly one of the characterizations of coprime numbers.
\end{proof}

When $k$ and $n$ are not coprime, even if the POVM is not IC, the map \eqref{eq:POVMcond} is injective on pure states.
In fact, we can strengthen the claim:

\begin{prop}
If $k<n$, fibers of the map \eqref{eq:POVMcond} which contain pure states are singletons.
\end{prop}

\begin{proof}
Observe that if $\rho=E_{kk}$, then
\begin{equation}\label{eq:cyclic}
\rho^P=\frac{1}{k}(\overbrace{1,\ldots,1}^{k\text{ times}},\overbrace{0,\ldots,0}^{n-k\text{ times}}) .
\end{equation}
The image of every other pure state is obtained from the formula above with a cyclic permutation, and they are all different if $n-k\geq 1$, i.e.~if at least one component is zero (if the last $1$ in $\rho^P$ is in position $i$, then necessarily $\rho=E_{ii}$).
Thus, the map \eqref{eq:POVMcond} is injective on pure states.

Now we wish to prove that if $\rho^P$ is a cyclic permutation of \eqref{eq:cyclic}, then the state $\rho$ is necessarily pure.
It is enough to show that, if $\rho$ is not pure, then $k\rho^P$ has (at least) one entry that is neither $0$ nor $1$.
If $\rho$ is not pure, it means that it has (at least) one entry $\lambda_{i_0}=\lambda$ which is neither $0$ nor $1$.
Thus
$$
k\rho^P(i_0+1)=\sum_{j=i_0+1}^{i_0+k}\lambda_j\leq \sum_{j\neq i_0}\lambda_j=1-\lambda< 1
$$
(since $k<n$, $\lambda_{i_0}$ does not appear in the former sum). If $\rho^P(i_0+1)\neq 0$ the proof is concluded. If $\rho^P(i_0+1)=0$, then $\lambda_j=0$ for all $i_0<j\leq i_0+k$. Therefore $k\rho^P(i_0)=\lambda_{i_0}=\lambda$ is neither $0$ nor $1$.
\end{proof}

\begin{samepage}
\begin{center}
\textsc{Acknowledgements}
\end{center}
We thank Alain Connes and Walter van Suijlekom for useful discussions.
FL and FDA acknowledge support form the INFN Iniziativa Specifica GeoSymQFT. FL~acknowledges financial support from the State Agency for Research of the Spanish Ministry of Science
and Innovation through the ``Unit of Excellence Maria de Maeztu 2020--2023'' award to the Institute of Cosmos Sciences (CEX2019-000918-M) and from PID2019-105614GB-C21 and
2017-SGR-929 grants.
GL acknowledges partial support from INFN, Iniziativa Specifica GAST,
from INdAM-GNSAGA and from the INDAM-CNRS IRL-LYSM.
\end{samepage}

\end{document}